\documentclass[12pt,leqno]{article}

\usepackage[cp1250]{inputenc}
\usepackage[OT4]{fontenc}
\usepackage{amsfonts}
\usepackage{amsmath}
\usepackage{amssymb}
\usepackage{amsthm}
\usepackage{fancyhdr}
\newcommand{\R}{\mathbb{R}}
\newcommand{\Z}{\mathbb{Z}}

\newcommand{\Apis}{{\cal A}}

\newcommand{\codim}{\mathop{\rm codim}\nolimits}
\newcommand{\rank}{\mathop{\rm rank}\nolimits}

\newcommand{\signature}{\mathop{\rm signature}\nolimits}
\newcommand{\sgn}{\mathop{\rm sgn}\nolimits}
\newcommand{\ar}{\longrightarrow}
\newcommand{\inv}{^{-1}}

\newtheorem{theorem}{Theorem}[section]
\newtheorem{lemma}[theorem]{Lemma}
\newtheorem{cor}[theorem]{Corollary}
\newtheorem{prop}[theorem] {Proposition}

\theoremstyle{definition}

\newtheorem{rem}[theorem]{Remark}
\newtheorem{ex}[theorem]{Example}

\title{\thanks{%
Iwona Krzy\.{z}anowska and Aleksandra~Nowel\\
University of Gda\'{n}sk,
              Institute of Mathematics \\
              80-952 Gda\'{n}sk, Wita Stwosza 57\\
              Poland\\
              Tel.: +48-58-5232059\\
              Fax: +48-58-3414914\\
              Email: Iwona.Krzyzanowska@mat.ug.edu.pl\\
              Email: Aleksandra.Nowel@mat.ug.edu.pl\\ \\
{\em Keywords:} Stiefel manifold; cross--cap; quadratic form\\              
2000 \emph{Mathematics Subject Classification} 14P25, 57R45, 12Y05}
Mappings into the~Stiefel manifold and~cross--cap singularities
}

\author{Iwona~Krzy\.{z}anowska \and Aleksandra~Nowel}

\date{July 2015}

\begin{document}

\def\nothanksmarks{\def\thanks##1{\protect\footnotetext[0]{\kern-\bibindent##1}}}

\nothanksmarks

\maketitle

\pagestyle{fancy}

\lhead{\fancyplain{}{\textsc{\small Krzy\.{z}anowska, Nowel}}}
\rhead{\fancyplain{}{\emph{\small Mappings into the~Stiefel manifold}}}

\begin{abstract}
Take $n>k>1$ such that $n-k$ is odd. In this paper we consider a mapping $a$ from $(n-k+1)$--dimensional closed ball into the space of $(n\times k)$--matrices such that its restriction to a sphere goes into the Stiefel manifold $\widetilde{V}_k(\R^n)$. We construct a homotopy invariant $\Lambda$ of $a|S^{n-k}$ which defines an isomorphism between $\pi_{n-k}\widetilde{V}_k(\R^n)$ and $\Z_2$. It can be used to calculate in an effective way the class of $a|S^{n-k}$ in $\pi_{n-k}\widetilde{V}_k(\R^n)$ for a polynomial mapping $a$ and to find the number mod $2$ of cross--cap singularities of a mapping from a closed $m$--dimensional ball into $\R^{2m-1}$, $m$ even. 
\end{abstract}

\section{Introduction}

Mappings from a sphere into the Stiefel manifold are natural objects of study. We will denote by $\widetilde{V}_k(\R^n)$ the non--compact Stiefel manifold (the set of all $k$--frames in $\R^n$). It is well--known (see \cite{hatcher}) that $\pi_{n-k}\widetilde{V}_k(\R^n)$ is isomorphic to $\Z_2$ if $n-k$ is odd and $k>1$, and to $\Z$ in all other cases.

Take $n>k>1$ such that $n-k$ is odd. In this paper we consider a mapping $a$ from $\overline{B}^{n-k+1}$ into $M_k(\R^n)$ --- the space of $(n\times k)$--matrices such that its restriction to the $(n-k)$--dimensional sphere goes into the Stiefel manifold. With $a$ we associate the mapping $\widetilde{a}\colon S^{k-1}\times \overline{B}^{n-k+1}\longrightarrow \R^n$ as 
$\widetilde{a}(\beta,x)=\beta_1a_1(x)+\ldots+\beta_ka_k(x)$ where 
$\beta=(\beta_1\ldots ,\beta_k)\in S^{k-1}$. 

Using $\widetilde{a}$ we construct a homotopy invariant $\Lambda$ of $a|S^{n-k}$ in the following way. If $\widetilde{a}\inv (0)$ is an infinite set then we can slightly perturb the map $a$ to get finite number of zeros of $\widetilde{a}$, and then
\[
\Lambda(a|S^{n-k} )=\sum _{(\beta, x)}\deg _{(\beta, x)}\widetilde{a} \mod 2, 
\]
where $(\beta, x)$ runs through half of the zeros of $\widetilde{a}$, i.e. we choose only one from each pair $(\beta, x), (-\beta, x)\in \widetilde{a} \inv (0)$.

It turns out that $\Lambda$ defines an isomorphism between $\pi_{n-k}\widetilde{V}_k(\R^n)$ and $\Z_2$ (Theorem \ref{isomorphismcg}). 

The case where $n-k$ is even was investigated in \cite{krzyzszafran}. The authors constructed an isomorphism between $\pi_{n-k}\widetilde{V}_k(\R^n)$ and $\Z$ as one--half the topological degree of $\widetilde{a}$. In the case where $n-k$ is odd this degree equals $0$ and the method from \cite{krzyzszafran} cannot be used.

Using the invariant $\Lambda$ and tools from \cite{krzyzszafran, szafraniec1}, for a polynomial mapping $a$ one can represent the class of $a|S^{n-k}$ in $\pi_{n-k}\widetilde{V}_k(\R^n)$ in terms of signatures of some quadratic forms or signs of determinants of their matrices. This provide an effective way to compute this invariant (see the algorithm and examples in Section \ref{computing}).

Moreover we present a nice characterisation of a mapping $a\colon B^{n-k+1} \longrightarrow M_k(\R^n)$ being transversal to $M_k(\R^n)\setminus \widetilde{V}_k(\R^n)$. We show that it happens if and only if zero is a regular value of $\widetilde{a}$ (Theorem \ref{lambdatransversal}). This result leads to another application of $\Lambda$. 

A mapping $f$ from an $m$--dimensional manifold $M$ into $\R^{2m-1}$ has a cross-cap at $p\in M$ if and only if locally near $p$ it has the form 
$(x_1,\ldots , x_m)\mapsto(x_1^2,x_2,\ldots , x_m,x_1x_2,\ldots ,x_1x_m)$ (see \cite[Theorem 4.6]{golub},  \cite[Lemma 2]{whitney1}).

In \cite{whitney1}, for $m$ even, Whitney proved that if $M$ is closed and $f$ has only cross--caps as singularities then the number of cross--caps is even. If  $M$ has a boundary then following \cite[Theorem 4]{whitney1}, for a homotopy $f_t\colon M\longrightarrow \R^{2m-1}$ regular in some open neighbourhood of $\partial M$, if the only singular points of $f_0$ and $f_1$ are cross--caps then the numbers of cross--caps of $f_0$ and $f_1$ are congruent mod $2$. In the case where $m$ is odd, one can associate signs with cross--caps, and to get similar results we have to count the sum of signs of cross--caps (see \cite{whitney1}). This case was investigated in \cite{krzyz}. 

We show (Corollary \ref{crosscaps}) that if $m$ is even, $f\colon \R^m\longrightarrow \R^{2m-1}$ is smooth, and for some $r>0$ there is no singular point of $f$ belonging to the sphere $S^{m-1}(r)$, then the number of cross--caps in $B^{m}(r)$ of $f$ mod $2$ can be expressed as $\Lambda(\alpha)$, where $\alpha$ is some mapping associated with $f$. In the polynomial case one can calculate the number of cross--caps of $f$ mod $2$ using algebraic methods. 

\section{Mappings into the Stiefel manifold $\widetilde{V}_k(\R^n)$ for $n-k$ odd} \label{stiefel}

By $B^n(p,r)$ we will denote the $n$--dimensional open ball centered at $p$, with radius $r$, by $\overline{B}^n(p,r)$ --- its closure, and by $S^{n-1}(p,r)$ --- the $(n-1)$--dimensional sphere. When we omit $p$ that means that the center is at the origin, if $r$ is omitted then $r=1$.

If $M$ is a smooth oriented $n$--manifold, $p\in M$, and $f\colon M\longrightarrow \R^n$ is such that $p$ is isolated in $f\inv (0)$, then there exists a compact $n$--manifold $N\subset M$ with boundary such that $f\inv (0)\cap N=\{ p\}$ and $f\inv (0)\cap \partial N=\emptyset$. Then by $\deg _p f$ we will denote the local topological degree of $f$ at $p$, i.e. the topological degree of the mapping $\partial N\ni x \mapsto f(x)/|f(x)|\in S^{n-1}$. 

If $g\colon M\longrightarrow \R^n$ is close enough to $f$, then $g\inv (0)\cap \partial N$ is also empty, and the topological degree of the mapping $\partial N\ni x \mapsto g(x)/|g(x)|\in S^{n-1}$ is equal to $\deg _p f$.

Let $n-k> 0$ be an odd number and $k>1$. Let us denote by $\widetilde{V}_k(\R^n)$ the Stiefel manifold, i.e. the set of all $k$--frames in $\R^n$. It is known (see \cite{hatcher}), that $\pi_{n-k}\widetilde{V}_k(\R^n)\simeq\Z_2$.

We write $M_k(\R^n)$ for the set of all $k$--tuples of vectors in $\R^n$ (i.e. the set of $(n\times k)$--matrices). We can consider $\widetilde{V}_k(\R^n)$ as a subset of $M_k(\R^n)$. 

Let $\alpha =(\alpha_1,\ldots ,\alpha_k):S^{n-k}\longrightarrow \widetilde{V}_k(\R^n)$ be smooth. Let us assume that there exists a smooth mapping $a=(a_1,\ldots, a_k)\colon \overline{B}^{n-k+1} \longrightarrow M_k(\R^n)$  (there $a_i(x)\in \R^n$) such that the restriction  $a|S^{n-k}$ is equal to $\alpha$.

We can define the mapping $\widetilde{a}\colon S^{k-1}\times \overline{B}^{n-k+1}\longrightarrow \R^n$ as $$\widetilde{a}(\beta,x)=\beta_1a_1(x)+\ldots+\beta_ka_k(x),$$ where $\beta=(\beta_1,\ldots ,\beta_k)\in S^{k-1}$ and $x=(x_1,\ldots ,x_{n-k+1})\in \overline{B}^{n-k+1}$. Then $\widetilde{a} \inv (0)\subset S^{k-1}\times B^{n-k+1}$ and $\widetilde{a}(\beta, x)=-\widetilde{a}(-\beta, x)$.

Let us assume that $\widetilde{a}\inv (0)$ is finite. Then we can define 
\[
\Lambda(\alpha )=\sum _{(\beta, x)}\deg _{(\beta, x)}\widetilde{a} \mod 2,
\]
where $(\beta, x)\in \widetilde{a} \inv (0)$ and we choose only one from each pair $(\beta, x), (-\beta, x)\in \widetilde{a} \inv (0)$ ($(\beta, x)$ runs through half of the zeros of $\widetilde{a}$). 

\begin{lemma}
$\Lambda$ is well defined, i.e. it does not depend on the choice of the mapping $a$ such that $a|S^{n-k}=\alpha$ and $\widetilde{a}\inv (0)$ is finite. 
\end{lemma}

This will be a simple consequence of Theorem \ref{lambdahomotopy} which is stated below.

Using properties of the topological degree and the fact that $\widetilde{a}(\beta, x)=-\widetilde{a}(-\beta, x)$ it is easy to show the following.

\begin{lemma} \label{bliskie}
Let mappings $a,b \colon \overline {B}^{n-k+1}\longrightarrow M_k(\R^n)$ be smooth such that $a|S^{n-k}, b|S^{n-k}\colon S^{n-k}\longrightarrow \widetilde{V}_k(\R^n)$ and $\widetilde{a}\inv (0)$, $\widetilde{b}\inv (0)$ are finite. If $a$ and $b$ are close enough to each other, then
\[
\Lambda (a|S^{n-k})=\Lambda (b|S^{n-k}) \mod 2.
\]
\end{lemma}

Put $\Sigma _k=\Sigma _k(\R^n)=M_k(\R^n)\setminus \widetilde{V}_k(\R^n)$. The set $\Sigma _k$ is algebraic and closed. There exists a natural stratification $(\Sigma _k^i)_{i=1,\ldots ,k}$ of $\Sigma _k$, where $\Sigma _k^i$ is a set of $(n\times k)$--matrices of rank $k-i$. According to \cite[Proposition II.5.3]{golub} $\Sigma _k^i$ is a smooth submanifold of $M_k(\R^n)$ of codimension $(n-k+i)i$. By \cite[Theorem II.4.9, Corollary II.4.12]{golub} in the set of smooth mappings $\overline{B}^{n-k+1}\longrightarrow M_k(\R^n)$ the subset of mappings transversal to $\Sigma_k$ (i.e. transversal to all $\Sigma _k^i$) is dense.

Since $\codim \Sigma _k^1=n-k+1$, for $a\colon \overline{B}^{n-k+1}\longrightarrow M_k(\R^n)$ transversal to $\Sigma _k$ ($a \pitchfork \Sigma _k$ for short), we obtain that $a\inv (\Sigma_k)=a\inv (\Sigma_k^1)$ is a finite set (see \cite[Proposition II.4.2, Theorem II.4.4]{golub}), moreover $\widetilde{a}\inv (0)$ is also finite.

Note that if $a\colon \overline{B}^{n-k+1}\longrightarrow M_k(\R^n)$ is such that $a(S^{n-k})\subset \widetilde{V}_k(\R^n)$, then according to \cite[Corollary II.4.12]{golub} we can find a smooth mapping transversal to $\Sigma _k$ which is arbitrarily close to $a$ and its restriction to $S^{n-k}$ equals $a|S^{n-k}$.

In Section \ref{proofs} we shall prove the following facts:

\begin{theorem}\label{lambdahomotopy}
Let $n-k$ be odd, $k>1$. Let $\alpha, \beta \colon S^{n-k}\longrightarrow \widetilde{V}_k(\R^n)$ be smooth mappings and assume that there exist smooth 
$a,b\colon \overline{B}^{n-k+1}\longrightarrow M_k(\R^n)$ such that $a|S^{n-k}=\alpha$, $b|S^{n-k}=\beta$, and $\widetilde{a},\widetilde{b}$ have a finite number of zeros. If $\alpha$ and $\beta$ are homotopic (i.e. $[\alpha ]=[\beta]$ in $\pi_{n-k}\widetilde{V}_k(\R^n)$), then
\[
\Lambda(\alpha)=\Lambda(\beta).
\]
\end{theorem}

\begin{theorem} \label{lambdatransversal}
Let  $a\colon \overline{B}^{n-k+1}\longrightarrow M_k(\R^n)$ be smooth and such that $a(S^{n-k})\subset \widetilde{V}_k(\R^n)$. Then the mapping $a$ is transversal to $\Sigma _k$ if and only if the origin is a regular value of $\widetilde{a}$. If this is the case and $n-k$ is odd, $k>1$,  then
\[
\Lambda (a|S^{n-k})=\# a\inv (\Sigma _k) \mod 2 .
\]
\end{theorem}

\begin{rem}
It is worth to underline that the above equivalence is true for arbitrary $n-k>0$. 
\end{rem}

Up to now we have defined $\Lambda (\alpha)$ for smooth $\alpha \colon S^{n-k}\longrightarrow \widetilde{V}_k(\R^n)$ with smooth $a\colon \overline{B}^{n-k+1}\longrightarrow M_k(\R^n)$ such that
$\alpha=a|S^{n-k}$, having useful properties stated above. 

By the Tietze Extension Theorem for any continuous mapping $\alpha \colon S^{n-k}\longrightarrow \widetilde{V}(\R^n)$ there exists a continuous mapping $f \colon \overline{B}^{n-k+1}\longrightarrow M_k(\R^n)$ such that $f|S^{n-k}=\alpha$. Then by \cite[Lemma 1.5]{hirsch} $f$ is homotopic to some smooth mapping from $\overline{B}^{n-k+1}$ to $M_k(\R^n)$, and so it is homotopic to a smooth $a \colon \overline{B}^{n-k+1}\longrightarrow M_k(\R^n)$ such that $a(S^{n-k})\subset \widetilde{V}_k(\R^n)$ and $\widetilde{a}\inv (0)$ is finite (such an $a$ can be chosen from the set of mappings transversal to $\Sigma_k$). 

Thus we can extend our definition of $\Lambda$ to continuous mapping $\alpha \colon S^{n-k}\longrightarrow \widetilde{V}_k(\R^n)$ in the following way:
\[
\Lambda(\alpha )=\sum _{(\beta, x)}\deg _{(\beta, x)}\widetilde{a} \mod 2,
\]
where $(\beta, x)\in \widetilde{a} \inv (0)$ and we choose only one from each pair $(\beta, x), (-\beta, x)\in \widetilde{a} \inv (0)$.

As a consequence of Theorem \ref{lambdahomotopy} in a natural way we obtain:

\begin{theorem}\label{lambdahomotopycg}
Let $\alpha, \beta \colon S^{n-k}\longrightarrow \widetilde{V}_k(\R^n)$ be continuous. If $\alpha$ and $\beta$ are homotopic (i.e. $[\alpha ]=[\beta]$ in $\pi_{n-k}\widetilde{V}_k(\R^n)$), then
\[
\Lambda(\alpha)=\Lambda(\beta).
\]
\end{theorem}

\begin{ex}\label{elneutralny}
Let $v_1,\ldots ,v_k$ be an orthonormal system of vectors in $\R^n$. For $a \colon \overline{B}^{n-k+1}\longrightarrow M_k(\R^n)$ given by $a(x)=(v_1,\ldots ,v_k)$ we obtain $a(\overline{B}^{n-k+1})\subset \widetilde{V}_k(\R^n)$. It is obvious that
\begin{itemize}
\item[(i)] $a\inv (\Sigma _k)=\emptyset$, hence $a \pitchfork \Sigma _k$, 
\item[(ii)] $\widetilde{a}\inv (0)=\emptyset$.
\end{itemize}
Each of the two above facts implies $\Lambda(a|S^{n-k})=0$.
\end{ex}

\begin{ex}
Let us define a mapping $a=(a_1,\ldots , a_k) : \overline{B}^{n-k+1}\longrightarrow M_k(\R^n)$ by $a_i(x)=(0,\ldots , 1, 0, \ldots ,0)$ where $1$ is on the $i$--th place, for $i=1,\ldots ,k-1$, and $a_k=(0,\ldots ,0,x_1,\ldots , x_{n-k+1})$. Then $a\inv (\Sigma _k)=\{ (0,\ldots, 0)\}$ and so $\alpha :=a|S^{n-k}\colon S^{n-k}\longrightarrow \widetilde{V}_k(\R^n)$. We have 
\[\widetilde{a}(\beta, x)=(\beta_1,\ldots , \beta_{k-1},\beta_kx_1,\ldots, \beta_kx_{n-k+1})\in \R^{n}\]
for $\beta\in S^{k-1}$, $x\in \overline{B}^{n-k+1}$.

Note that $\widetilde{a}\inv (0)=\{(0,\ldots ,0,1;0,\ldots ,0), (0,\ldots ,0,-1;0,\ldots ,0)\}$. 
Thus
\[
\Lambda(\alpha )= \deg _{(0,\ldots ,0,1;0,\ldots ,0)}\widetilde{a} \mod 2.
\]
It is easy to see that $(0,\ldots ,0,\pm 1;0,\ldots ,0)$ are regular points of $\widetilde{a}$, so $a \pitchfork \Sigma _k$ and $\deg _{(0,\ldots ,0,1;0,\ldots ,0)}\widetilde{a}=\pm 1$. Each of the two facts implies $\Lambda(\alpha)= 1 \mod 2$.
\end{ex}

Using the above Examples and Theorems \ref{lambdahomotopy} and \ref{lambdatransversal} one can easily show the following.

\begin{theorem} \label{isomorphismcg}
The mapping $\pi_{n-k}\widetilde{V}_k(\R^n)\ni \sigma \mapsto \Lambda (\alpha)\in \Z _2$, where $\alpha\colon S^{n-k}\longrightarrow \widetilde{V}_k(\R^n)$ is such that $[\alpha ]=\sigma$, is an isomorphism.
\end{theorem}

It is obvious that in all the results of this Section, instead of $\overline{B}^{n-k+1}$ we can use $\overline{B}^{n-k+1}(r)$, for any $r>0$.

\section{Proofs} \label{proofs}

\begin{lemma}
Let $a\colon \overline{B}^{n-k+1}\longrightarrow M_k(\R^n)$ be a smooth mapping such that $a(S^{n-k})\subset \widetilde{V}_k(\R^n)$. For any $x\in a\inv (\Sigma _k)$ there exists such a diffeomorphism $\Phi \colon M_k(\R^n) \longrightarrow M_k(\R^n)$ that $\Phi (a(x))$ has a form 
\[
\left [
\begin{array}{c|c}
0 &  \cr
\vdots & \mathbf{0}_{(n-k+1)\times (k-1)} \cr
0 &  \cr
\hline
0 &  \cr
\vdots & * \cr
0 &  \cr
\end{array}
\right ],
\]
moreover $\Phi (\Sigma _k^r)=  \Sigma _k^r$ for $r=1,\ldots ,k$, and so $a \pitchfork \Sigma _k$ at $x$ if and only if $(\Phi \circ a) \pitchfork \Sigma _k$ at $x$. Here $\mathbf{0}$ denotes the appropriate zero matrix.
\end{lemma}

We can always choose such a $\Phi$ to be a composition of elementary row and columns operations. For this $\Phi$ we have:

\begin{lemma} \label{transformations}
There exists a diffeomorphism $\Psi \colon S^{k-1} \longrightarrow S^{k-1}$ such that 
\[
\widetilde{a}(\beta, x)=0 \ \Leftrightarrow \ \widetilde{\Phi(a)}(\Psi (\beta), x)=0,
\]
and if $(\beta, x)\in \widetilde{a}\inv (0)$,  then $(\beta, x)$ is a regular point of $\widetilde{a}$ if and only if $(\Psi (\beta),x)$ is a regular point of $\widetilde{\Phi (a)}$. 
\end{lemma}
\begin{proof}
It is obvious that $\Phi$ can be taken as a composition of simple column and row transformations. Hence it suffices to show the conclusion of the Lemma for any simple column or row transformation.

Let us assume that $\Phi$ is a simple column--multiplying transformation. For example let $\Phi$ multiply the first column by $c\neq 0$, so that $\Phi (a)=(ca_1,a_2,\ldots ,a_k)$.

Then we can define $\Psi \colon S^{k-1} \longrightarrow S^{k-1}$ as
\[
\Psi (\beta)=\Psi (\beta _1,\ldots ,\beta _k)={\frac{(\frac{1}{c}\beta _1, \beta _2,\ldots ,\beta _k)}{\|(\frac{1}{c}\beta _1, \beta _2,\ldots ,\beta _k)\|}_{eucl}}.
\]
Hence
\begin{equation} \label{ga}
\widetilde{\Phi(a)}(\Psi(\beta),x)=\frac{|c|}{\sqrt{\beta _1^2+c^2(1-\beta _1^2)}}\widetilde{a}(\beta,x)=:g(\beta)\widetilde{a}(\beta,x). 
\end{equation}
Let us observe that for $\beta \in S^{k-1}$, we have $\sqrt{\beta _1^2+c^2(1-\beta _1^2)}\neq 0$, of course $g(\beta)\neq 0$. So $\widetilde{a}(\beta, x)=0$ if and only if $\widetilde{\Phi(a)}(\Psi (\beta), x)=0$. 

Take $(\beta,x)\in \widetilde{a}\inv (0)$. Then by (\ref{ga}) we have $\dfrac{\partial \widetilde{\Phi (a)}}{\partial x_i}(\Psi (\beta), x)=g(\beta )\dfrac{\partial \widetilde{a}}{\partial x_i}(\beta, x)$. Then (see \cite[Section 5]{kanosza})
\[
\rank d\widetilde{\Phi (a)} (\Psi(\beta),x)= 
\]
\begin{equation} \label{macierz}
=1+\rank \left [
\begin{matrix}
\frac{g(\beta)}{c}\beta _1 & g(\beta)\beta _2 & \ldots & g(\beta)\beta _k & 0 & \ldots & 0\cr
 & & & & & & \cr
ca_1(x) & a_2(x) & \ldots & a_k(x) & g(\beta) \frac{\partial \widetilde{a}}{\partial x_1}(\beta,x) & \ldots & g(\beta) \frac{\partial \widetilde{a}}{\partial x_{n-k+1}}(\beta,x)\cr 
& & & & & & 
\end{matrix}
\right ].
\end{equation}
Since $\beta _1 a_1(x)+\ldots +\beta _ka_k(x)=0$, an easy computation shows that the number (\ref{macierz}) is equal to
\[
1+\rank \left [
\begin{matrix}
\beta _1 & \ldots & \beta _k & 0 & \ldots & 0\cr
 & & & & & & \cr
a_1(x) & \ldots & a_k(x) & \frac{\partial \widetilde{a}}{\partial x_1}(\beta,x) & \ldots & \frac{\partial \widetilde{a}}{\partial x_{n-k+1}}(\beta,x)\cr 
& & & & & & 
\end{matrix}
\right ]=
\rank d\widetilde{a} (\beta, x).
\]
Thus $(\beta, x)$ is a regular point of $\widetilde{a}$ if and only if $(\Psi (\beta),x)$ is a regular point of $\widetilde{\Phi (a)}$.

If $\Phi$ is any other simple column transformation the proof is similar.

If $\Phi$ is a simple row transformation, then there exists a diffeomorphism $\gamma \colon \R^n \longrightarrow \R^n$ such that $\gamma (0)=0$ and $\widetilde{a}(\beta, x) = \gamma \circ \widetilde{\Phi (a)}(\beta, x)$, so the conclusion in this case is also true. 
\end{proof}

\begin{proof}[Proof of Theorem \ref{lambdatransversal}]
Let us note that if $a\pitchfork \Sigma_k$, then $a\inv (\Sigma_k)=a\inv (\Sigma_k^1)$ is a finite set, and so $\widetilde{a}\inv(0)$ is finite.

On the other hand if the origin is a regular value of $\widetilde{a}$, then $\widetilde{a}\inv(0)$ is finite, so $a\inv (\Sigma_k)$ is also finite and  $a\inv (\Sigma_k)=a\inv (\Sigma_k^1)$.

Let us take $\bar{x}\in a\inv (\Sigma_k^1)$. According to Lemma \ref{transformations} we can assume that $a(\bar{x})$ has a form 
\[
a_1(\bar{x})=(0,\ldots,0) \mbox{ and } a_i(\bar{x})=(0,\ldots,0,a_i^{n-k+2}(\bar{x}),\ldots,a_i^{n}(\bar{x}))
\]
for $i=2,\ldots,k$ (here $a_i^j$ is the element standing in the $j$--th row and $i$--th column). Then 
\[
\rank[a_i^j(\bar{x})]=k-1,
\]
where $j=n-k+2,\ldots ,n$, $i=2,\ldots ,k$, and $\widetilde{a}(\bar{\beta},\bar{x})=0$ if and only if $\bar{\beta}=(\pm 1,0,\ldots ,0)$.

For a matrix $F=[f_i^j]_{k\times n}$ we will denote by $\hat{F}$ the submatrix $[f_i^j]$, where $j=n-k+2,\ldots ,n$, $i=2,\ldots ,k$. If $F$ is close enough to $a(\bar{x})$ in $M_k(\R^n)$ then $\det \hat{F}\neq 0$. In this case it is easy to check that $F\in \Sigma _k^1$ if and only if 
\[
f^j_1=[f^j_2\ \ldots \ f^j_k] \hat{F}\inv [f^{n-k+2}_1\ \ldots \ f^n_1]^T
\]    
for $j=1,\ldots ,n-k+1$ (see \cite[Lemma II.5.2]{golub}).

Hence the tangent space $T_{a(\bar{x})}\Sigma_k^1$ is spanned by vectors $v_i=(0,\ldots ,0,\ldots , 1, \ldots ,0)$, where $1$ stands at $(i+n-k+1)$--th place, $i=1,\ldots ,nk-(n-k+1)$.

Let us observe that $a\pitchfork \Sigma _k^1$ at $\bar{x}$ if and only if $\rank da(\bar{x})$ is maximal (i.e. equals $n-k+1$) and $T_{a(\bar{x})}\Sigma_k^1\cap da(\bar{x})\R^{n-k+1}=\{0\}$. It is equivalent to the condition:
\begin{equation} \label{maxrank}
\rank \frac{\partial (a_1^1,\ldots ,a_1^{n-k+1})}{\partial (x_1,\ldots ,x_{n-k+1})}(\bar{x})=n-k+1. 
\end{equation}

On the other hand $(\bar{\beta},\bar{x})$ is a regular point of $\widetilde{a}$ if and only if $\rank d\widetilde{a}(\bar{\beta},\bar{x})=n$. It is equivalent to (see \cite[Section 5]{kanosza})
\[
n+1 =\rank \left [
\begin{matrix}
2\beta _1 & \ldots & 2\beta _k & 0 & \ldots & 0 \cr
a_1^1(x) & \ldots & a_k^1(x) &  & & \cr
\vdots & & \vdots & & \dfrac{\partial \widetilde{a}}{\partial (x_1,\ldots ,x_{n-k+1})}(\beta,x) & \cr
a_1^n(x) & \ldots & a_k^n(x) &  & & \cr
\end{matrix}
\right ]_{(\bar{\beta},\bar{x})}
=
\]
\[
=\rank \left [
\begin{array}{cccc|ccc}
\pm 2 & 0 & \ldots & 0 & 0 & \ldots & 0 \cr
 & & \mathbf{0} & & & \dfrac{\partial (a_1^1,\ldots ,a_1^{n-k+1})}{\partial (x_1,\ldots ,x_{n-k+1})}(\bar{x}) & \cr
\hline 
0 & a_2^{n-k+2}(\bar{x}) & \ldots & a_k^{n-k+2}(\bar{x}) &  & & \cr
\vdots & \vdots & & \vdots & & * & \cr
0 & a_2^{n}(\bar{x}) & \ldots & a_k^{n}(\bar{x}) &  & & 
\end{array}
\right ].
\]
This holds if and only if (\ref{maxrank}) holds.
\end{proof}

\begin{proof}[Proof of Theorem \ref{lambdahomotopy}]
According to \cite[Theorem 19.2]{whitney3} we can find a finite refinement $(Q^r)$ of the stratification $(\Sigma _k^i)$ fulfilling Whitney (a) condition. Then by \cite[Theorem, p.~274]{trotman}, in the set of smooth mappings $\overline{B}^{n-k+1}\longrightarrow M_k(\R^n)$ the subset of mappings transversal to all $Q^r$ is not only dense, but also open.  

The codimension of the biggest strata of $(Q^r)$ equals $n-k+1$, so if $a$ is transversal to all $Q^r$ then it is transversal to all $\Sigma _k^i$.

According to Lemma \ref{bliskie} we can assume that $a$ and $b$ are transversal to all $Q^r$. 

Let $h\colon S^{n-k}\times [0;1]\longrightarrow \widetilde{V}_k(\R^n)$ be the homotopy between $\alpha$ and $\beta$. The mappings $a$, $b$ and $h$ define a continuous mapping from $\left (S^{n-k}\times [0;1]\right )\cup \left (\overline{B}^{n-k+1}\times \{ 0,1\}\right )$ into $M_k(\R^n)$. By the Tietze Extension Theorem it can be extend to $\overline{B}^{n-k+1}\times [0;1]$. So close enough to this mapping we can find a smooth $H\colon \overline{B}^{n-k+1}\times [0;1]\longrightarrow M_k(\R^n)$ which is transversal to $\Sigma _k$ (see \cite{golub, hirsch}). Since $\codim \Sigma _k^i> n-k+2 =\dim \left (\overline{B}^{n-k+1}\times [0;1]\right )$ for $i>1$, we have $H\pitchfork \Sigma _k^1$ and $H\inv (\Sigma _k^i)=\emptyset$ for $i>1$.

The mapping $H$ has the following properties:
\begin{itemize}
\item $H(\cdot,0)$, $H(\cdot,1)$ are close enough to $a$, $b$ (resp.); 
\item $H|(S^{n-k}\times [0;1])$ is close enough to $h$, so that it goes into $\widetilde{V}_k(\R^n)$ and $H|(S^{n-k}\times [0;1])\pitchfork \Sigma _k$.
\end{itemize}

Since the set of smooth mappings $\overline{B}^{n-k+1}\longrightarrow M_k(\R^n)$ transversal to all $Q^r$ is open, $H(\cdot,0)$, $H(\cdot,1)$ are transversal to all $Q^r$, and so to $\Sigma _k$.

We have $H\inv (\Sigma _k)=H\inv (\Sigma _k^1)$ and by \cite[Theorem, p.~60]{gupol} it is a compact $1$--dimensional manifold whose intersection with $\left (S^{n-k}\times [0;1]\right ) \cup \left ( \overline{B}^{n-k+1}\times \{0,1\}\right )$ is exactly its boundary. Because $H|(S^{n-k}\times [0;1])$ goes into $\widetilde{V}_k(\R^n)$, we have 
\[
\partial H\inv (\Sigma _k)=H\inv (\Sigma _k)\cap \left (\overline{B}^{n-k+1}\times \{0,1\}\right )=\]
\[
=\left (H(\cdot,0)\inv (\Sigma _k)\times \{0\}\right )\cup \left (H(\cdot,1)\inv (\Sigma _k)\times \{1\}\right ). 
\]
Since the boundary of a compact $1$--dimensional manifold has an even number of points, 
\[
0= \# \partial H\inv (\Sigma _k)=\# (H(\cdot,0))\inv (\Sigma _k)+ \# (H(\cdot,1))\inv (\Sigma _k) \mod 2.
\]

According to Theorem \ref{lambdatransversal} and Lemma \ref{bliskie} we have 
\[
\Lambda(\alpha )+\Lambda (\beta)=\# (H(\cdot,0))\inv (\Sigma _k)+ \# (H(\cdot,1))\inv (\Sigma _k)=0 \mod 2, 
\]
hence
\[
\Lambda(\alpha )=\Lambda (\beta) \mod 2. 
\]
\end{proof}

\section{Counting $\Lambda$ for polynomial mappings into the Stiefel manifold} \label{computing}

In \cite{krzyzszafran} the authors defined an isomorphism from $\pi_{n-k}\widetilde{V}_k(\R^n)$ to $\Z$ for $n-k$ even. They also presented a method to calculate values of this isomorphism for polynomial mappings. It is easy to observe that some of methods presented in \cite[Section 3]{krzyzszafran} can be used to obtain an effective way of computing $\Lambda$ in the polynomial case for $n-k$ odd.

Let $a=(a_1,\ldots ,a_k):\R^{n-k+1}\longrightarrow M_k(\R^n)$ 
be a polynomial mapping, $n-k$ odd, $k>1$. Denote by $\left [a_{i}^j(x)\right ]$ the $(n\times k)$--matrix given by $a(x)$ (here $a_i^j$ is the element standing in the $j$--th row and $i$--th column). Then
$$\widetilde{a}(\beta, x)=
\beta_1a_1(x)+\ldots +\beta_ka_k(x)=\left [a_{i}^j(x)\right ]\left [\begin{array}{c}
\beta_1\\
\vdots\\
\beta_k
\end{array} \right ]:S^{k-1}\times \R^{n-k+1}\longrightarrow \R^n.$$

By $I$ we will denote the ideal in $\R[x_1,\ldots ,x_{n-k+1}]$ generated by all 
$k\times k$ minors of  $\left [a_{i}^j(x)\right ]$. Let $V(I)=\{x\in\R^{n-k+1}\ |\  h(x)=0\mbox{ for all }h\in I\}$.

As in \cite[Lemma 3.1]{krzyzszafran}
we conclude that $p\in V(I)$ if and only if  $a_1(p), \ldots ,a_k(p)$ are linearly dependent, 
i.e. if  $\widetilde{a}(\beta, p)=0$ for some $\beta$.

Put 
$$m(x)=\det\left [\begin{array}{ccc}
a_{2}^1(x)&\ldots &a_{k}^1(x)\\ \\
a_{2}^{k-1}(x)&\ldots &a_{k}^{k-1}(x) 
\end{array} \right ],$$
and $\Apis= \R[x_1,\ldots ,x_{n-k+1}]/I$. Let us assume that $\dim \Apis <\infty$, so that $V(I)$ is finite. 

Assume that $\rank \left [a_{i}^j(x)\right ]\geqslant k-1$ for $x$ in $\R^{n-k+1}$. 
Then one can choose such coordinates in $\R ^{n-k+1}$ that $m\neq 0$ in $V(I)$. So $\widetilde{a}\inv (0)$ is finite and if $\widetilde{a}(\beta,x)=0$ then $\beta _1\neq 0$.

Let us define 
\[
F(\lambda,x)=(F_1,\ldots , F_n)(\lambda,x)=\left [a_{ij}(x)\right ]\left [\begin{array}{c}
1\\ \lambda_2\\
\vdots\\
\lambda_k
\end{array} \right ]:\R^{k-1}\times \R^{n-k+1}\longrightarrow\R^n
\]
where $\lambda=(\lambda _2,\ldots ,\lambda _k)$.

Take $(\beta,x)\in \widetilde{a}\inv (0)$, $\beta _1>0$, and $\lambda=(\beta _2/\beta _1,\ldots ,\beta _k/\beta _1)$.   According to \cite[Lemma 3.5]{krzyzszafran}, $(\lambda,x)$ is an isolated zero of $F$, moreover $\deg _{(\beta,x)}\widetilde{a}=\deg_{(\lambda,x)}F$. So we get:

\begin{prop} If $\alpha=a|S^{n-k}(r)\colon S^{n-k}(r)\longrightarrow \widetilde{V}_k(\R^n)$ then
\[
\Lambda(\alpha )=\sum _{(\lambda, x)}\deg _{(\lambda, x)}F \mod 2,
\]
where $x\in B^{n-k+1}(r)$ and $(\lambda, x)\in F\inv (0)$.
\end{prop}

For $h\in \Apis$ we will denote by $T(h)$ the trace of the linear endomorphism
\[
\Apis \ni a\mapsto h\cdot a\in \Apis.  
\]
Let $\Theta _{\delta}\colon \Apis \longrightarrow \R$ be a quadratic form given by $\Theta _{\delta}(a)=T(\delta \cdot a^2)$. 

Taking a polynomial map $\delta\colon \R ^n \longrightarrow \R$ as in \cite[Section 3]{krzyzszafran} and using \cite[Lemma 3.4]{krzyzszafran}, and the same arguments as in the proof of \cite[Theorem 3.3, Lemma 3.9]{krzyzszafran}, one can show the following.

\begin{prop}
Let $n-k$ be odd, $k>1$, $r>0$, $\omega(x)=r^2-x_1^2-\ldots -x_{n-k+1}^2$.
If $a=(a_1,\ldots ,a_k):\R^{n-k+1}\longrightarrow M_k(\R^n)$ 
is a polynomial mapping such that $\dim\Apis<\infty$, 
$I+\langle m \rangle=\R[x_1,\ldots ,x_{n-k+1}]$, and quadratic forms 
$\Theta_\delta,\, \Theta_{\omega\cdot\delta}:\Apis\longrightarrow\R$ are non--degenerate, then 
\[
\alpha=a|S^{n-k}(r)\colon S^{n-k}(r)\longrightarrow \widetilde{V}_k(\R^n),
\] 
and 
\[\Lambda(\alpha)=
\frac{1}{2}(\signature\Theta_{\delta}+\signature\Theta_{\omega \cdot\delta})\mod 2.
\]
\end{prop}

By \cite[Lemma 2.2, Theorem 2.3]{szafraniec1} we obtain two facts:

\begin{cor} Under the above assumptions
\[
\Lambda(\alpha)=\dim \Apis +1+\frac{1}{2}(\sgn \det [\Theta_{\delta}]+\sgn \det[\Theta_{\omega \cdot\delta}])\mod 2.
\]
\end{cor}

\begin{cor}
Let $\varphi \colon \Apis \ar \R$ be a linear functional, $\Phi$ and $\Psi$ be the bilinear symmetric forms on $\Apis$ given by $\Phi(f,g)=\varphi(fg)$ and $\Psi(f,g)=\varphi(\omega fg)$. If $\det [\Psi]\neq 0$, then  $\det [\Phi]\neq 0$ and
\[
\Lambda(\alpha)=\dim \Apis +1+\frac{1}{2}(\sgn \det [\Phi]+\sgn \det[\Psi])\mod 2.
\]
\end{cor}

\begin{ex}
Let $a,b\colon \overline{B}^{2}\longrightarrow M_2(\R^3)$ be given by

\[a(x,y)=\left (
\left [
\begin{matrix}
5x^2y+2y^2+3x+2\\
5xy^2+2x^2+5x+3\\
2x^3+4xy+2y+1
\end{matrix}
\right ],
\left [
\begin{matrix}
5x^2y+y^2+3x+3\\
y^3+2xy+3y+2\\
4x^3+x^2+3y+5\\
\end{matrix}
\right ]
\right )
\]

\[b(x,y)=\left (
\left [
\begin{matrix}
4xy^2+3x^2+y+5\\
5xy^2+5y^2+y+5\\
3x^2y+3x^2+x+2
\end{matrix}
\right ],
\left [
\begin{matrix}
y^3+4xy+y+1\\
y^3+x^2+4y+5\\
5y^3+5xy+5y+2\\
\end{matrix}
\right ]
\right )
\]

Using \textsc{Singular} \cite{singular} and the above facts one may check that: 
$a|S^{1}$, $a|S^{1}(2)$, $a|S^{1}(10)$, $b|S^{1}$, $b|S^{1}(\sqrt{7/2})$, $b|S^{1}(10)$ go into $\widetilde{V}_2(\R ^3)$, and $\widetilde{a}$, $\widetilde{b}$ have a finite number of zeros. 

Moreover, for $a$ the dimension of the algebra $\Apis$ equals $23$, for $b$ it equals $21$.

For the mapping $a$ we get
\[
\Lambda(a |S^{1}(1))=\Lambda(a |S^{1}(2))=1 \mod 2,
\] 
but 
\[\Lambda(a |S^{1}(10))=0 \mod 2,
\]
and for $b$
\[
\Lambda(b |S^{1}(1))=\Lambda(b |S^{1}(10))=0 \mod 2,\]
but
\[\Lambda \left (b |S^{1}\left ( \sqrt{7/2}\right ) \right )=1 \mod 2.
\]
\end{ex}

\section{The number of cross--caps}

Let $M$ be a smooth $m$--dimensional manifold.
According to  \cite{golub,whitney2,whitney1}, a point $p\in M$ is a cross--cap of a smooth mapping $f:M\longrightarrow \R^{2m-1}$, if there is a coordinate system near $p$, such that 
in some neighbourhood of $p$ the mapping $f$ has the form 
\[(x_1,\ldots , x_m)\mapsto(x_1^2,x_2,\ldots , x_m,x_1x_2,\ldots ,x_1x_m).\]

Take $f:M\longrightarrow \R^{2m-1}$ with only cross--caps as singularities. For $m$ even, 
by \cite[Theorem 3]{whitney1}, if $M$ is a closed manifold, then $f$ has an even number of cross--caps.   

Let $(M,\partial M)$ be $m$--dimensional smooth compact manifold with boundary. Take a continuous mapping $f\colon [0;1]\times M\longrightarrow \R^{2m-1}$ such that there exists a neighbourhood of $\partial M$ in which all the $f_t$'s are regular, and $f_0$, $f_1$ have only cross--caps as singularities. According to  \cite[Theorem 4]{whitney1}, if $m$ is even, then $f_0$ and $f_1$ have the same number of cross--caps $\mod 2$.

Let $f\colon \R^m\longrightarrow \R^{2m-1}$ be smooth.
Using  an equivalent definition \cite[Definition VII.4.5]{golub} of a cross--cap and the fact that the space of $1$--jets of mappings from $\R^m$ to $\R^{2m-1}$ coincides with $\R^m\times \R^{2m-1}\times M_m(\R^{2m-1})$, it is easy to see that:

\begin{rem} The point $p\in \R^m$ is a cross--cap of $f$ if and only if $\rank df(p)=m-1$ and $df \pitchfork \Sigma _m^1(\R^{2m-1})$ at $p$ (where we consider $df$ as a map going to $M_m(\R^{2m-1})$). 
\end{rem}

Let us assume that for some $r>0$ there is no singular points of $f$ belonging to $S^{m-1}(r)$. Then $(df)|S^{m-1}(r)\colon S^{m-1}(r)\longrightarrow \widetilde{V}_m(\R^{2m-1})$. 
By Theorem \ref{lambdatransversal} we obtain the following facts.

\begin{rem}
The mapping $f|\overline{B}^{m}(r)$ has only cross--caps as singular points if and only if the origin is a regular value of $\widetilde{df}\colon S^{m-1}\times \overline{B}^{m}(r)\longrightarrow \R ^{2m-1}$. 
\end{rem}

If $m$ is even, then the difference $2m-1-m=m-1$ is odd and $\Lambda ((df)|S^{m-1}(r))$ is defined. 

\begin{cor} \label{crosscaps}
If $m$ is even and $f$ has only cross--caps as singular points, then 
\[
\Lambda ((df)|S^{m-1}(r))=\mbox{ number of cross--caps of } f \mbox{ in } \overline{B}^{m}(r) \mod 2 .
\]
\end{cor}

\begin{cor} \label{close}
If $m$ is even then the number of cross--caps in $\overline{B}^{m}(r)$ of every smooth mapping $g\colon \R^m\longrightarrow \R^{2m-1}$ close enough to $f$ with only cross--caps as singular points is congruent to  $\Lambda ((df)|S^{m-1}(r)) \mod 2$.
\end{cor}

Using computer system \textsc{Singular} \cite{singular} we apply the facts from Sections \ref{stiefel} and \ref{computing} to present some examples illustrating the above Corollaries.

\begin{ex}
Let $f\colon \R^2 \longrightarrow \R^3$ be given by
\[
f(x,y)=(15xy^3+19y^3+9x^2+6y,25y^3+15x^2,7y^3+21xy).
\]
One can check that $f$ has only cross--caps as singular points, in fact there are $3$ of them, and
the number of cross--caps of $f$ in $\overline{B}^2(1)$ and in $\overline{B}^2(10)$ is congruent to $1$, but in $\overline{B}^2(5)$ is congruent to $0$ modulo $2$ (Corollary \ref{crosscaps}).
\end{ex}

\begin{ex}
Let $f\colon \R^2 \longrightarrow \R^3$ be given by
\[
f(x,y)=(2xy^3+7x^2y,6xy^5+29x^4y+20y^4+26y^3+27x^2+9x,
\]
\[
21x^2y^4+7x^2y^3+11xy^3+20xy^2+10xy+8y).
\]
One can check that $f$ has only cross--caps as singular points, in fact there are $14$ of them, and
the number of cross--caps of $f$ in $\overline{B}^2(1)$ is congruent to $1$, but in $\overline{B}^2(1/10)$ and in $\overline{B}^2(10)$ is congruent to $0$ modulo $2$ (Corollary \ref{crosscaps}).
\end{ex}

\begin{ex}
Let $f\colon \R^2 \longrightarrow \R^3$ be given by
\[
f(x,y)=(21x^2y^2+13xy^2+7y^2+27y,16xy^4+7y^4+19x^3,7xy^4+6x^3y+21x^2y).
\]
In this case $f$ has some singular points that are not cross--caps. One can check that $f$ has no singular points on the sphere $S^1(1)$ and every smooth mapping close enough to $f$ with only cross--caps as singular points has an odd number of cross--caps in $\overline{B}^2(1)$ (Corollary \ref{close}).
\end{ex}


\begin{thebibliography}{99}

\bibitem{singular} {G.--M.~Greuel, G.~Pfister, H.~Sch\"onemann}, {{\sc Singular} 3.0.2}. A Computer Algebra System for Polynomial Computations;

\bibitem{golub} M.~Golubitsky, V.~Guillemin, Stable mappings and their singularities, 1973 by Springer--Verlag New York;

\bibitem{gupol} {V.~Guillemin, A.~Pollack,} Differential topology, Prentice-Hall, Inc., Englewood Cliffs, N.J., 1974; 

\bibitem{hatcher} A.~Hatcher, Algebraic Topology, Cambridge University Press, 2002;

\bibitem{hirsch} M.~W.~Hirsch, Differential topology, Graduate Texts in Mathematics, no. 33. Springer--Verlag, New York--Heidelberg, 1976; 

\bibitem{kanosza} {I.~Karolkiewicz, A.~Nowel, Z.~Szafraniec,} An algebraic formula for the intersection number of a polynomial immersion, Journal of Pure and Applied Algebra 214 (2010), 269-280;

\bibitem{krzyz} {I.~Krzy\.{z}anowska}, Cross--cap singularities counted with sign, 
arXiv:1506.04873v1 [math.AG] (submitted on 16 Jun 2015);

\bibitem{krzyzszafran} {I.~Krzy\.{z}anowska, Z.~Szafraniec}, Polynomial mappings into a Stiefel manifold and immersions, Houston J. Math. 40 (2014), no. 3, 987--1006;

\bibitem{szafraniec1} Z.~Szafraniec, Topological degree and quadratic forms, Journal of Pure and Applied Algebra 141, 1999, p. 299--314;

\bibitem{trotman} D.~J.~A.~Trotman, Stability of transversality to a stratification implies Whitney (a)--regularity, Invent. Math. 50 (1978/79), no. 3, 273--277;

\bibitem{whitney3} {H.~Whitney}, Tangents to an analytic variety, Ann. of Math. (2) 81, 1965 496--549;

\bibitem{whitney2} {H.~Whitney}, The general type of singularity of a set of $2n-1$ smooth functions of $n$ variables, Duke Math. J. 10, (1943). 161--172;

\bibitem{whitney1} {H.~Whitney}, The Singularities of a Smooth $n$--Manifolds in $(2n-1)$--Space, Ann. of Math. (2) 45, (1944), 247--293;

\end{thebibliography}
\end{document}